\documentclass[11pt]{elsarticle}
\usepackage[latin1]{inputenc}
\usepackage{amsmath,amsthm, amsfonts,amssymb}
\usepackage{graphicx}
\usepackage{color}
\usepackage{epsfig}
\input epsf

\usepackage[top=1.4in, bottom=1.4in, left=1.4in, right=1.4in]{geometry}
\begin{document}

\title{Modeling battery cells under discharge\\ using kinetic and
  stochastic battery models}
%[Modeling battery cells]

\author[mathUU]{Ingemar Kaj\corref{cor1}}
\ead{ikaj@math.uu.se}

\address[mathUU]{Department of Mathematics, Uppsala University\\ P.O. Box 480,
SE  751 06 Uppsala, Sweden}
%ikaj\@@\,math.uu.se}

\author[isp,bf]{Victorien Konan\'e}
\ead{konane@math.uu.se}
\address[isp]{International Science Program, Uppsala University}
\address[bf]{Department of Mathematics, University of Ouagadougou, Burkina Faso}

\begin{abstract}
In this paper we review several approaches to mathematical modeling of
simple battery cells and develop these ideas further with emphasis on
charge recovery and the response behavior of batteries to given
external load. We focus on models which use few parameters and basic
battery data, rather than detailed reaction and material
characteristics of a specific battery cell chemistry, starting with
the coupled ODE linear dynamics of the kinetic battery model.  We show
that a related system of PDE with Robin type boundary conditions
arises in the limiting regime of a spatial kinetic battery model, and
provide a new probabilistic representation of the solution in terms of
Brownian motion with drift reflected at the boundaries on both sides
of a finite interval. To compare linear and nonlinear dynamics in
kinetic and stochastic battery models we study Markov chains with
states representing available and remaining capacities of the battery.
A natural scaling limit leads to a class of nonlinear ODE, which can
be solved explicitly and compared with the capacities obtained for the
linear models. To indicate the potential use of the modeling we
discuss briefly comparison of discharge profiles and effects on
battery performance.

\end{abstract} 


\maketitle

\def\wt{\widetilde}  
\def\R{\mathbb R} 
\def\bbe{\mathbb E} 
\def\S{\mathbb S}
\def\bbp{\mathbb P}  
\def\D{{\mathcal D}} 
\def\X{{\mathcal X}}
\def\H{{\mathcal H}}  
 
\def\supp{\rm supp}
\def\diam{\rm diam}  

\newtheorem{theorem}{Theorem}
\newtheorem*{theorem*}{Theorem}
\newtheorem{corollary}{Corollary}
\newtheorem{lemma}{Lemma}
\newtheorem{proposition}{Proposition}
\newtheorem{definition}{Definition}

\section{Introduction}

The subject of this work is mathematical modeling of state-of-charge 
in simple battery cells, such as a non-rechargeable
3 Volts Lithium coin battery. The goal is to understand the response
of the battery, and ultimately to predict battery lifetime, as
energy is consumed under a given discharge usage
pattern.  The main incentive for our work is the battery usage in
Wireless Sensor Networks and similar Internet-of-Things systems.
These networks consist of inter-connected low-cost nodes, equipped
with basal sensors, computer, radio and a battery, expected to run for
many years under very low intensity loads and short dutycycles.
Within the tight cost restrictions typical of such
  systems, methods or techniques providing ``battery-charge
  indicators'' do not seem to be within reach currently.  To make
progress in this direction it is essential to address the problem of
predicting battery life.  Our paper intends to cover some of the
required modeling groundwork and discuss analysis of linear, nonlinear
and stochastic aspects of modeling battery cells.

Mathematical modeling of batteries has developed over several decades
along with the growth of new battery technologies and materials.  Yet,
there has been relatively little in-depth study of widely available,
inexpensive coin cell batteries and on special load characteristics
including short load periods.  Primarily, lithium and lithium-ion
battery models have been developed within electrochemical engineering.
Two recent survey and review works \cite{land2013,ram2012} represent
the state-of-art of modeling based on the fundamental principles of
electrochemistry, and emphasize the wide range of scales involved.
The temporal and spatial scales of the physics and chemistry of the
battery range from macroscopic level all the way down to the atomistic
level.  The tutorial review by Landstorfer and Jacob \cite{land2013}
provides a framework of non-equilibrium thermodynamics as the
foundation for studying the electrode, electrolyte and interface
reactions in great detail.  The review work by Ramadesigan {\it et
  al.} \cite{ram2012}, summarizes the literature on such models and, in
addition, brings a systems engineering approach applied to Li-ion
batteries.  This type of model can be said to begin with the
pseudo-two-dimensional (P2D) model of Doyle {\it et
  al.} \cite{doy1993}, which leads to a coupled system of non-linear
PDEs. More generally, coupled systems of equations with complex
boundary conditions are derived, which connect charge concentrations
with transport and kinetics of reactant species.  Several approaches
have been proposed to simplify the resulting sets of equations and
allow for numerical computations, see e.g.\ \cite{sub2009,dao2012}.

In this work we use the different mathematical approach to battery
modeling developed in communication engineering for computer science
applications, see e.g. Jongerden and Haverkort \cite{jon2009}.  Where
chemical engineering modeling typically begins with a detailed scheme
of reactions and mechanisms in the various phases and interfaces of
the cell, these models view the battery as a generic device subject to
some fundamental principles. The main focus of the modeling changes
and is now rather the response of the battery to external load.  A
typical purpose is load scheduling to optimize battery utilization.
Important aspects of battery behavior from this point of view are the
rate-capacity effect and charge recovery. Quoting \cite{roh2013}: {\it
  The former refers to the fact that a lower discharge rate is more
  efficient than a higher; more charge can be extracted from the
  battery before reaching a given cut-off value. The latter refers to
  the fact that an intermittent discharge is more efficient than a
  continuous one. Because of these effects, different battery loads
  that use the same total charge do not result in the same device
  lifetime.}  The most basic of these methods use linear ODEs
\cite{man1993,man1994} and gradually build complexity by using PDEs
and other means \cite{rak2003,rao2003,jon2009}.  A further direction
is stochastic modeling using Markov chain dynamics \cite{chi,cra}.

The aim of our work is to investigate to what extent these simplified
battery models are able to capture important aspects of battery
behavior, and to get new insights by developing the mathematical
models further.  Of special interest is the essential response of the battery
to deterministic or random on-off discharge patterns typical for
batteries in wireless sensor networks, and the ability of a cell to
recover charge during operation.
Charge recovery is believed to depend on a number of internal
mechanisms, such as convection, diffusion, migration, and
charge-transfer.   
In the modeling work we put special emphasis on separating the roles
of diffusion, which is the motion of electroactive species in the presence of a
concentration gradient, and migration, the motion of charged species in
an electric field.   

Our starting point is the kinetic battery model of Manwell and McGowen
\cite{man1993,man1994}, which describes the joint evolution of
available charge and bound charge over time. Charge recovery in this
framework consists in the continuous transition of bound charge to
available charge.  As observed in \cite{jon2009} and further
investigated in \cite{kajkon2012}, more general spatial versions of
these models are related to the diffusion model of Rakhmatov and
Vrudhula \cite{rak2003,rao2003}, and leads to a class of second order
diffusion equations with Robin type boundary conditions.  Quoting
\cite{raoetal2005}, the linear dynamics of the kinetic battery model
is {\it useful in getting an intuitive idea of how and why the
  recovery occurs but it needs a number of additions to be useful for
  the types of batteries used in mobile computing}. Using experimental
data from Ni-MH batteries, \cite{raoetal2005} proposed a modified,
non-linear, factor in the flow charge and discussed related stochastic
versions of the model. In an effort to compare more systematically
linear and nonlinear dynamics in kinetic and stochastic battery
models, we study discrete time Markov chains with nonlinear jump
probabilities derived from a simplified charge transport scenario. By
a scaling approximation we obtain in the limit a deterministic
nonlinear ODE, with explicit solutions which, in principle, can be
compared to those of the linear approach. The unifying
  mathematical aspect in our analysis is the representation of
  remaining and available capacities as time-autonomous systems.

In Section 2, following preliminaries on capacities,
  internal charge recovery and discharge profiles, we consider the
kinetic battery model and discuss a variation.  Then we set up an
extended version of the spatial kinetic battery model with a finite
number of serial compartments, derive the spatially continuous
limiting PDE, and give a probabilistic representation of the solution
of the PDE in terms of Brownian motion with drift reflected at the
boundaries on both sides of a finite interval. The solution represents
the capacity storage of a battery and these tools allow us to study
the balance of available and remaining stored capacity.
Section 3 studies a nonlinear, stochastic Markov chain
  model and its deterministic ODE approximation.  With proper choice
of nonlinear dynamics for charge recovery due to transfer, diffusion
and migration effects, we then propose a somewhat wider class of
nonlinear ODE of potential use for battery modeling.  As a
consequence, it is possible to study performance measures such as
battery life, delivered capacity and gain, and to compare and optimize
the performance of batteries.

\section{Linear Battery Models}

\subsection{Nominal, theoretical and available capacity}
\label{seq:nomcap}
We consider a primary (non-rechargeable) battery cell consisting of two
electrodes, anode and cathode, linked by an electrolyte. The cell
contains a certain amount of chemically reactive material which is
converted into electrical energy by an oxidation reaction at the
anode. Primary lithium batteries have a lithium anode and may have
soluble or solid electrolytes and cathodes.  The mass of material
involved in the battery reaction yields a higher concentration of
electrons at the anode, and hence by Faraday's first law the transfer
of a proportional quantity of electrical charge.  This determines a
terminal voltage between the pair of electrodes. By closing a wired
circuit between the terminals a current of electrons will start moving
through the wire from anode to cathode where they react with a
positively charged reactant, manifesting the ability of the battery to
drive electric current.  The intensity of the current depends on the
total resistance along the wire.  Inside the battery the movement of
charge-carriers forms a corresponding ionic current, which is
controlled by a variety of mechanisms, among them migration of ions,
diffusion of reactant species, and charge-transfer reaction.
Migration is generated by an electric potential gradient (electric
field) and convective diffusion by the concentration
gradient. Conductivity arises from the combination of migration and
diffusion.  Charge-transfer reactions take place when migrating ions
are transferred from the electrolyte bulk through the anode surface.
 
It is sometimes helpful to keep track of proper units.  The battery
has a given voltage $E_0$ in volts [V] and a theoretical capacity $T$
in ampere hours [Ah], representing the entire storage of chemically
reactive material in the cell. We write $N$ [Ah] for the nominal
capacity of a fully loaded cell, where $N\le T$. This is
  the amount of electric charge which is delivered if the cell is put
  under constant, high load and drained until a predefined cut-off
  voltage is reached.  Measuring time $t$ in hours [h] we write
$\Lambda(t)$ for the consumed capacity [Ah] and $v(t)=T-\Lambda(t)$
for the remaining capacity [Ah], at time $t$. Here,
$\Lambda=(\Lambda(t))_{t\ge 0}$ is an increasing function, typically
continuous with the slope representing intensity of the current. It is
also plausible to let the discharge function have jumps to be
interpreted as spikes of charge units being released point-wise to the
device driven by the battery.  As an additional level of generality it
is straightforward to consider $\Lambda$ defined on a probability
space and subject to a suitable distributional law of a random
process.  The instantaneous discharge current [A] at $t$ is the
derivative $\Lambda'(t)=\lim_{h\to }h^{-1}\Lambda(t+h)$ and the
average discharge current [A] is the quantity
$\bar\lambda=\lim_{t\to\infty}t^{-1}\Lambda(t)$, assuming this limit
exists.

We are interested in the behavior of the battery cell when 
exposed to the accumulated discharge process $\Lambda$, in
particular regarding
\begin{align*}
u(t)&=\mbox{available capacity [Ah] at time $t$}, \quad u(0)=N\\[1mm]
\widetilde u(t)&=u(t)/N= \mbox{state of charge at time $t$},\quad
\widetilde u(0)=1\\[1mm]
E(t)&=\mbox{voltage [V] at time $t$},\quad E(0)=E_0. 
\end{align*}
While there is no obvious method of observing state-of-charge
empirically, voltage is accessible to measurements at least in
principle.  To describe typical voltage, we imagine that a fully
charged battery at time $t=0$ is connected to a closed circuit at
constant discharge current $\delta=\Lambda'(t)|_{t=0}$ [A].  The
result is an instant voltage drop from $U_0$ to a new level at
approximate voltage $U_0-\delta r$, where $r$ is an internal
resistance [ohm] of the cell. As long as charge is consumed, the
state-of-charge will then begin to decline over time accompanied by a
subsequent change of voltage. If after a period of discharge the
current is disconnected and the battery temporarily put to rest, then
the voltage increases. First, instantly, by the amount $\delta r$ and
then over the course of the off-period at some rate due to recovery
effects inside the cell.  The resulting voltage versus time curve
extended over a longer time span would typically stay nearly constant
or exhibit slow decline over most of the active life of the battery
followed by a steeper decent until a cut-off level $E_\mathrm{cut}$ is
reached, beyond which the cell is considered to be non-operational.
To relate state of charge and voltage we recall that the actual load a
time $t$ is given by $\Lambda'(t)$ and apply the equilibrium Nernst
equation, \cite{newman2004} Ch.\ 1, to obtain
\begin{equation}\label{voltagecapacity}
E(t)=E_0-\Lambda'(t)\,r+K_e \ln(\widetilde u(t)),\quad K_e=\frac{RT_a}{zF}, 
\end{equation}
where $R$ is the ideal gas constant, $T_a$ is absolute temperature,
$z$ is the valency of the battery reactant ($z=1$ for Lithium), $F$ is
Faraday's constant, and dimensions are such that $K_e$ is measured in
volts. The internal resistance, however, may have a more complex
origin arising from a series of resistances in the electrodes and
electrolyte. Electrical circuit models use equivalent
  electrical circuits to capture such current-voltage
  characteristics. In this direction, \cite{kimqiao2011} consider a
  hybrid model of kinetic battery and electrical circuit models to
  fascilitate the derivation of voltage in terms of exponential and
  polynomial functions of state of charge. Building on this approach,
  \cite{roh2013} fits empirical battery data in order to evaluate
  battery models for wireless sensor network applications.

\subsection{Internal charge recovery}

The general principle for charge recovery is that of balancing the
discharge rate by a positive drift of the available
  capacity due to the release and transport of stored charges.  Such
effects should exist as long as the theoretical capacity of the cell
has not yet been fully consumed, that is $v(t)\ge 0$. The
first recovery mechanism to take into account is (solid-state)
diffusion of charge carriers caused by the build-up of a concentration
gradient in the electrolyte during discharge. The drift of the process
is convective flow and a diffusion coefficient controls random
variations around the main direction of transport.  Diffusion
transport of charge carriers might be a slow process which persists
even if the load is removed and the battery put to rest, and runs
until charge concentrations have reached local equilibrium.  Another
mechanism for gaining capacity due to recovery is migration of
charge-carriers caused by the electric field, as an action of a
potential gradient. The strength of this effect should increase with
the gap $N-u(t)$ between maximal and actual capacity.  It appears
reasonable to assume that the effect of migration is ongoing whether
the battery is under load or at rest. The final aspect of recovery we
wish to include in the modeling scenario is charge transfer, meaning
the transfer of charges from electrolyte through an interface to the
terminal electrode.  A simplified approach for this effect is that of
a friction mechanism, such that a fraction of recovered charges are
actually effectuated proportional to the applied load current, either
instantaneous current or average current over long time.

\subsection{Discharge profiles}\label{sec:dischargeprofile}  

The battery models we study are introduced in relation to an arbitrary
accumulated discharge function $\Lambda$.  To engage in a more
detailed analysis of battery performance we consider three stylized
examples of $\Lambda$, which represent typical discharge patterns for
the intended usage of the battery.

\noindent
{\it Constant current.} The first such pattern is that
of draining the battery at a constant current $\lambda$ which remains
the same over the entire battery life until the cell is emptied. Clearly,
$\Lambda(t)=\lambda t$ and $\bar\lambda=\lambda$.  

\noindent
{\it Deterministic on-off pattern.} The second example is relevant for
the case when we know in advance both the amount of work the battery
is supposed to power and the scheduled timing of loads. For such cases
we consider a deterministic pulse-train which consists of a periodic
sequence of cycles of equal length.  Each cycle begins with an active
on-period during which a pulse of constant load is transmitted,
followed by an off-period of rest and no load.  Specifically we assume
that a current $\delta$ [A] is drawn continuously during each
on-period of length $\tau_\mathrm{on}$ followed by a dormant
off-period of length $\tau_\mathrm{off}$.  Hence the cycle duration is
$\tau=\tau_\mathrm{on}+\tau_\mathrm{off}$ and the dutycycle is given
by the fraction $q=\tau_\mathrm{on}/\tau$.  We introduce
\[
J_t=\sum_{j=0}^\infty 1_{\{j\tau\le
  t<j\tau+\tau_\mathrm{on}\}},\quad t\ge 0,
\]
so that $J_t=1$ if $t$ belongs to an on-period and $J_t=0$ for $t$ in
an off-period.  Then the consumed capacity is 
\[
\Lambda(t)=\delta \int_0^t J_s\,ds,\quad t\ge 0,
\]
and $(\Lambda(t))$ is a piecewise continuous function
with non-decreasing rate $\Lambda(dt)=\delta J_t\,dt$. 
The average discharge rate equals $\bar\lambda= \delta q$.

\noindent
{\it Random discharge pattern.}
Our third stylized example of discharge mechanisms represents the 
case where no information except average load is available in advance
of battery operation.  In this situation we consider completely random
discharge with the load to be drawn from the battery per time unit
scattered independently and uniformly random in the sense of the
Poisson process.  Here we take $\Lambda(t)=\delta \tau_\mathrm{on}
N^{(1/\tau)}_t$, where $(N_t^{(\lambda)})_{t\ge 0}$ denotes a standard
Poisson process on the half line with constant intensity $\lambda>0$.
This amounts to saying that the battery is drained from energy in
small jumps of charge $\delta \tau_\mathrm{on}$ which occur interspaced by
independent and exponentially distributed waiting times with expected
value $\tau$. Again the average discharge current is $\bar\lambda=
\ell\tau_\mathrm{on}/\tau=\delta q$.

\subsection{Kinetic battery model}

The Kinetic Battery Model, originally introduced for lead acid
batteries in \cite{man1993,man1994}, takes the view that remaining
capacity of the battery is split in two wells, or compartments, one
representing available charge and the other bound
charge. Discharge is the consumption of available charge
and charge recovery is the flow of matter from the bound well to the
available one.  The available capacity $u(t)$ in this
model is precisely the amount of charge in the available well as
function of time. Hence we call $y(t)=v(t)-u(t)$, the bound charge and
decompose remaining capacity as $v(t)=u(t)+y(t)$, $u(0)=N$,
$y(0)=T-N$.  With the use of the fraction $c=N/T$, $0<c<1$, the two
wells are assigned a measure of height given by $u(t)/c$ and
$y(t)/(1-c)$.  The principle of the kinetic battery model is that
bound charge becomes available at a rate which is proportional to the
height difference $y(t)/(1-c)-u(t)/c$.  Once available, no charges
return to the bound state. Thus,
\begin{equation}
\left\{
\begin{array}{lll}
du(t)
&=-\Lambda(dt)+k\Big(\frac{y(t)}{1-c}-\frac{u(t)}{c}\Big)\,dt,\quad &u(0)=N\\[2mm]
dy(t) &=-k\Big(\frac{y(t)}{1-c}-\frac{u(t)}{c}\Big)\,dt,\quad &y(0)=T-N,
\end{array}
\right. \label{kibameqnsyst}
\end{equation}
where $k>0$ is a reaction parameter.   By assumption, 
$v(t)=u(t)+y(t)=T-\Lambda(t)$, $t\ge 0$. It is convenient therefore to consider
the pair $(v(t),u(t))$. With $k_c=k/c(1-c)$ as an
alternative parameter, 
\begin{equation}\label{kibameqn}
du(t)=-\Lambda(dt)+k_c(cv(t)-u(t))\,dt,\quad u(0)=N.
\end{equation}
In (\ref{kibameqn}),  charge recovery as represented by the factor
$cv(t)-u(t)=N-u(t)-c\Lambda(t)$ may be seen as the combination of
a migration effect due to the term $N-u(t)$ together with
(negative) drift $-c\Lambda(t)$. To clarify these connections, Ref.\
\cite{kajkon2012} studied a reweighted version of the model.  In
this paper we wish to develop these ideas further and hence consider
the closely related reweighted model
\begin{equation}\label{kibameqnplus}
du(t)=-\Lambda(dt)+k_c(cv(t)-u(t)+p(N-u(t))\,dt,\quad u(0)=N,
\end{equation}
where the parameter $p$ controls charge recovery due to
migration. The case $p\ge 0$ reflects migration of
  additional charge from the bound well adding to the internal charge
  recovery. The case $p\le 0$ represents loss of recovery due to
  migration.
%\begin{equation}
%du(t)=-\Lambda(dt)+\kappa(N-u(t))\,dt-c\kappa\Lambda(t)\,dt,\quad u(0)=N,
%\end{equation}
The linear system (\ref{kibameqnplus}) is readily solved as
\begin{equation}\label{kibameqnplussol}
u(t)=N- c \Lambda(t)-(1-c)\int_0^t e^{-k_c(1+p)(t-s)}\,\Lambda(ds). 
\end{equation}
For later reference we note that the relevant version of
(\ref{kibameqnsyst}) for this extended case is 
\begin{equation} \label{kibameqnsystplus}
\left\{
\begin{array}{lll}
du(t)
&=-\Lambda(dt)+k_c(cy(t)-(1-c)u(t)+p(N-u(t)))\,dt, \\[2mm]
dy(t) &=-k_c(cy(t)-(1-c)u(t)+p(N-u(t)))\,dt.
\end{array}
\right. 
\end{equation}
We emphasize that the discharge profile $\Lambda(t)$ is arbitrary for
this version of the kinetic battery model. For example, with the random
discharge pattern discussed in section \ref{sec:dischargeprofile} the
integral in (\ref{kibameqnplussol}) is a stochastic
Poisson integral and the the corresponding solution is a
  random process $U(t)$, 
which may be written
\[
U(t)=N- \delta\tau_\mathrm{on}\sum_{s_i\le t} 
(c+(1-c)e^{-k_c (1+p)(t-s_i)}), 
\]
where the sum extends over all jumps $s_i$ in $[0,t]$ of a Poisson
process with intensity $1/\tau$.

Of course, the kinetic battery model could only give a crude
indication of the processes behind real battery behavior.  
In contrast, the diffusion battery model of Rakhmatov
  and Vrudhula introduced in \cite{vrud2001} is an approach closer to
electrochemical modeling, detailing reaction kinetics and internal
transport mechanisms.  Both types of models are relevant for
engineering battery performance.  A significant step towards
unifying these approaches is the extension to spatial
versions of the kinetic battery model, where charges move inside of a
reservoir of bound charge according to the same local dynamics as the
simplest case above.

\subsection{Spatial kinetic battery model}

The diffusion battery model \cite{vrud2001,rak2003},
views battery operation as the diffusion of ions between two
electrodes separated by an electrolyte consisting of a linear region
  $[0,\ell]$ of the one-dimensional line. Initially, the electroactive
  species are uniformly distributed over space. Over time, the
  concentration $C(t,x)$ of electroactive species at time $t\ge 0$ and
  distance $x$ from the electrode at $x=0$ develop following Fick's
  laws of diffusion with suitable boundary conditions at both
  electrode endpoints $x=0$ and $x=\ell$.  In
  \cite{jon2009}, comparing kinetic and diffusion battery models,
  Jongerden and Haverkort introduced the idea of placing a finite
number of charge compartments in series along the spatial range and
letting charges move between adjacent components according to Eq.\
(\ref{kibameqnsyst}).  Discharge occurs at the anode which is located
in one end point of the spatial interval. Starting from a state of
fully charged compartments a spatial charge profile develops over time
and determines the pace at which the battery is drained. 
  This system is shown to match the diffusion model when the diffusion
  equations for $C(t,x)$ are discretized over an equally sized
  partition of $[0,\ell]$. In this sense the diffusion model is a
  continuous version of the kinetic battery model.
In this section we present further developments of this theory,
building on previous work in Ref.\ \cite{kajkon2012}.  Indeed, we
analyze the effects of charge recovery in the spatial setting and
provide in explicit form the available charge and other
performance measures in terms of basic parameters of diffusion and
migration.

We begin by considering a battery cell consisting of $m$ adjacent
fluid compartments and a function $u(t)=(u_1(t),\dots,u_m(t))$ which
gives the charge content in each component over time.  Here $u_1$ is
the available charge, $u_2$ is a bound well charge for $u_1$ and so on
until $u_m$, which is a bound well charge for $u_{m-1}$.  By letting
Eqn.\ (\ref{kibameqnsyst}) act pairwise on adjacent compartments, we
obtain the coupled system of linear equations
\[
\left\{
\begin{array}{ll}
du_1(t)=-\Lambda(dt)+k_c(cu_2(t)-(1-c)u_1(t))\,dt\\
du_2(t)=-k_c (cu_2(t)-(1-c)u_1(t))\,dt+k_c (cu_3(t)-(1-c)u_2(t))\,dt\\
\qquad \vdots\\
du_{m-1}(t)=-k_c
(cu_{m-1}(t)-(1-c)u_{m-2}(t))\,dt \\
\hskip 3.5 cm +k_c(cu_m(t)-(1-c)u_{m-1}(t))\,dt\\
du_m(t)=-k_c(cu_m(t)-(1-c)u_{m-1}(t))\,dt.
\end{array} \right.
\]
In greater generality, the reweighted version
(\ref{kibameqnsystplus}) yields
\[
\left\{
\begin{array}{ll}
du_1(t)=-\Lambda(dt)+k_c(cu_2(t)-(1-c)u_1(t))\,dt+
k_cp(N-u_1(t))\,dt\\
du_2(t)=-k_c(cu_2(t)-(1-c)u_1(t))\,dt+k_c
(cu_3(t)-(1-c)u_2(t))\,dt\\
\hskip 3.5 cm -k_cp(u_2(t)-u_1(t))\,dt\\
\qquad \vdots\\
du_{m-1}(t)=-k_c
(cu_{m-1}(t)-(1-c)u_{m-2}(t))\,dt \\
\hskip 2 cm +k_c(cu_m(t)-(1-c)u_{m-1}(t))\,dt-k_cp(u_{m-1}(t)-u_{m-2}(t))\,dt\\
du_m(t)=-k_c(cu_m(t)-(1-c)u_{m-1}(t))\,dt-k_cp(N-u_{m-1}(t))\,dt.
\end{array} \right.
\]
To see more clearly the structure in this system of equations, we
introduce the parameter $\mu_c=2c-1$ and employ the notations 
\[
\nabla u_k(t)=u_{k+1}(t)-u_k(t),\quad 
\Delta u_k(t)=u_{k-1}(t)-2u_k(t)+u_{k+1}(t).
\]
Then
\[
\left\{
\begin{array}{lc}
du_1(t)=-\Lambda(dt)+k_c\Big(\frac{1}{2}\nabla u_1(t)
  +\mu_c(u_1(t)+u_2(t))/2
+p(N-u_1(t))\Big)\,dt \\
du_2(t)=k_c\Big(\frac{1}{2} \Delta u_2(t)
+\mu_c(\nabla u_2(t)+\nabla u_1(t))/2
-p\nabla u_1(t)\Big)\,dt\\
\qquad \vdots\\
du_{m-1}(t)=k_c\Big(\frac{1}{2}\Delta u_{m-1}(t)
+\mu_c(\nabla u_{m-1}(t)+\nabla u_{m-2}(t))/2
 -p\nabla u_{m-2}(t)\Big)\,dt\\
du_m(t)=-k_c\Big(\frac{1}{2}\nabla u_{m-1}(t)
+\mu_c(u_{m-1}(t)+u_m(t))/2+p(N-u_{m-1}(t))\Big)\,dt.
\end{array} \right.
\]
Here, the parameter $p$ controls drift originating from
the terms proportional to $N-u(t)$ in (\ref{kibameqnsystplus}). In
particular, the nominal capacity $N$ remains as a parameter in
the boundary equations for $u_1$ and $u_m$.

\subsection{Limiting PDE, continuous space}

Our next goal is to identify a limiting partial differential equation
for the charge concentration profile $u(t)$ in the limit $m\to\infty$
as the size of the charge compartments tends to zero and the number of
wells goes to infinity.  The case $p=0$ is studied in Ref.\
\cite{kajkon2012} and we will use a similar method for the general
setting. 

Let $\ell=(T-N)/N$ and consider the strip $0\le x\le \ell$ partitioned
in $m$ equal intervals of length $\varepsilon=\ell/m$.  For $x=j\varepsilon$,
$j=1,\dots,m$, we define $u_\varepsilon(t,x)=u_j(t)$ and note that 
\[
\nabla
u_\varepsilon(t,x)=u_\varepsilon(t,x+\varepsilon)-u_\varepsilon(t,x)
\]
and
\[
\Delta
u_\varepsilon(t,x)=u_\varepsilon(t,x-\varepsilon)-2u_\varepsilon(t,x)+u_\varepsilon(t,x+\varepsilon). 
\]
To match spatial and temporal scaling we introduce the scaled parameters
$\kappa=k/m^2$, $\kappa_c=\kappa/c(1-c)$. The relations
derived above for $u_j$, $2\le j\le m-2$ imply, for $x\in
\{2/m,\dots,(\ell-1)/m\}$,
\begin{equation*}
du_\varepsilon(t,x)=\kappa_c \Big(\ell^2\frac{1}{2}
\frac{\Delta u_\varepsilon(t,x)}{\varepsilon^2}
+\ell B_\varepsilon u_\varepsilon(t,x)\Big)\,dt
\end{equation*}
where $B_\varepsilon$ is the linear drift operator
\begin{equation*}
B_\varepsilon u_\varepsilon(t,x)=m\mu_c
\frac{\nabla u_\varepsilon(t,x)+\nabla u_\varepsilon(t,x-\varepsilon)}{2\varepsilon}
-mp\frac{\nabla u_\varepsilon(t,x-\varepsilon)}{\varepsilon}. 
\end{equation*}
The additional relations for $u_1$ and $u_m$ correspond to 
boundary equations for $u_\varepsilon$, which attain the form
\begin{align*}
&\frac{du_\varepsilon(t,\varepsilon)}{m}
= -\frac{\Lambda(dt)}{m}  %\\
%&\quad 
+\frac{\kappa_c}{2} \left\{\ell\, 
\frac{\nabla u_\varepsilon(t,\varepsilon)}{\varepsilon}
+m\mu_c(u_\varepsilon(t,\varepsilon)+u_\varepsilon(t,2\varepsilon))
+2mp(N-u_\varepsilon(t,\varepsilon)) \right\}dt
\end{align*}
and 
\begin{align*}
&\frac{du_\varepsilon(t,\ell)}{m}= %\\&
-\frac{\kappa_c}{2} \left\{\ell\, 
\frac{\nabla u_\varepsilon(t,\ell-\varepsilon)}{\varepsilon}
+m\mu_c(u_\varepsilon(t,\ell-\varepsilon) +u_\varepsilon(t,\ell)) 
+2mp(N-u_\varepsilon(t,\ell-\varepsilon))\right\}dt
\end{align*}
Based on the above relations for the system of $m$ compartments one
can see that in order to balance all terms in the limit $m\to\infty$,
it is natural to introduce two drift parameters $\mu$ and $\rho$, both
of arbitrary sign, and replace $c$ by $c_m=(1+\mu/m)/2$ and $p$ by
$p_m=\rho/m$. Then, for large $m$,
\[
c_m\sim 1/2,\quad \kappa_{c_m}\sim 4\kappa, \quad m\mu_{c_m}\sim
\mu,\quad mp_m\sim \rho,
\]
where $\kappa>0$ is a reaction parameter, $\mu$ a diffusion parameter
and $\rho$ a migration parameter.
This gives the approximate system
\begin{eqnarray*}
du_\varepsilon(t,x)=
-\Lambda(dt)\delta_\varepsilon(dx)+2\kappa\ell^2 \frac{\Delta
u_\varepsilon(t,x)}{\varepsilon^2}\,dt
+4\kappa\ell(\mu-\rho)\frac{\nabla u_\varepsilon(t,x)}{\varepsilon}\,dt
\end{eqnarray*}
with Robin type boundary conditions
\[
\ell\, 
\frac{\nabla u_\varepsilon(t,\varepsilon)}{\varepsilon}
+2\mu u_\varepsilon(t,\varepsilon)
+2\rho(N-u_\varepsilon(t,\varepsilon))=0
\]
and
\[ 
\ell\,
\frac{\nabla u_\varepsilon(t,\ell-\varepsilon)}{\varepsilon}
+2\mu u_\varepsilon(t,\ell) 
+2\rho(N-u_\varepsilon(t,\ell))=0.
\]
We conclude that the relevant limiting equation in the limit
$\varepsilon=1/m\to 0$, is defined on $0\le x\le \ell~$ by 
\begin{eqnarray}\nonumber
&&du(t,x)=-\Lambda(dt)\delta_0(dx) +2\kappa\ell^2
\frac{\partial^2 u}{\partial x^2}(t,x)\,dt
+4\kappa\ell(\mu-\rho) \frac{\partial u}{\partial x}(t,x)\,dt,
\\\nonumber
&& \quad \ell\, \frac{\partial u}{\partial x}(t,0+)=-2\mu u(t,0)
-2\rho(N-u(t,0))\\
&&\quad \ell\, \frac{\partial u}{\partial x}(t,\ell-)
=2\mu u(t,\ell)+2\rho(N-u(t,\ell)),\,\qquad
u(0,x)=u_0(x). 
\label{pdefinal}
\end{eqnarray}
Here, $u(t,0)_{t\ge 0}$ is the charge density at the boundary of the
battery and $\{u(t,x),0<x<\ell\}_{t\ge 0}$ is the fluid level of a
reservoir of bound charge such that $\int_{(0,\ell)} u(t,x)\,dx$ is
what remains in the reservoir at time $t$.

\subsection{Probabilistic solution}

To state a probabilistic representation of the solution to
(\ref{pdefinal}), let $(B_t)_{t\ge 0}$ denote Brownian
motion with variance parameter $4\kappa\ell^2$ and constant drift
$-4\kappa\ell\delta$.  The parameter $\delta$ in the drift of the
Brownian motion corresponds to $\delta=\mu-\rho$ in
(\ref{pdefinal}).  
We assume that $(B_t)$ is confined to the interval $(0,\ell)$ and
subject to reflecting boundaries at both end points $0$ and $\ell$.
Let $p_{\ell,\delta}(t,y,x)$ be the transition density of $(B_t)$ so
that $P(B_t\in dx|\xi_0=y)=p_\ell(t,y,x)\,dx$.  We will use a
spectral type representation for $p_{\ell,\delta}(t,y,x)$ known to be
\begin{eqnarray}\label{transitionprob}
&&p_{\ell,\delta}(t,y,x)=\frac{2\delta}{\ell}\frac{e^{-2\delta
      x/\ell}}{1-e^{-2\delta}}
  +\frac{2e^{-\delta(x-y)/\ell}}{\ell}\times\\ &&\sum_{n=1}^\infty
  (\cos(\frac{n\pi x}{\ell})-\frac{\delta}{n\pi}\sin(\frac{n\pi
    x}{\ell})) (\cos(\frac{n\pi
    y}{\ell})-\frac{\delta}{n\pi}\sin(\frac{n\pi y}{\ell}))
  \frac{e^{-2\kappa(\delta^2+n^2\pi^2)t}}{1+(\delta/n\pi)^2}.
\nonumber
\end{eqnarray}
The above expression is derived in Ref.\ \cite{svensson1990} and
discussed and compared with an alternative representations in
Ref.\ \cite{veestraeten2004}.  In particular, for the symmetric case,
letting $\delta\to 0$,
\[
p_{\ell,0}(t,y,x)=\frac{1}{\ell}+\frac{2}{\ell} \sum_{n=1}^\infty \cos(n\pi
x/\ell)\cos(n\pi y/\ell)\, e^{-2\kappa n^2\pi^2 t}.
\]
Asymptotically, for $0\le y\le \ell$ as $t$ tends to infinity,
\begin{equation}\label{homogeneousasymptotic}
\lim_{t\to\infty}p_{\ell,\delta}(t,y,x)=\frac{2\delta}{\ell}\frac{e^{-2\delta
      x/\ell}}{1-e^{-2\delta}},\quad 
\lim_{t\to\infty}p_{\ell,0}(t,y,x)=\frac{1}{\ell},\quad 0\le x\le \ell.
\end{equation}

\begin{theorem}\label{thm:pde}
Suppose that the spatial kinetic battery model with reaction parameter
$\kappa>0$, migration parameter $\rho$, and diffusion parameter
$\mu$ is defined on an interval $\ell=(T-N)/N>0$ where
$N$ is the nominal capacity and $T$ the theoretical
capacity. Let $\delta=\mu-\delta$. A given nonnegative
function $\{u_0(y),\, 0\le y\le \ell\}$, is the initial charge profile
of the battery and $\Lambda(dt)$ is a given discharge pattern.  We
restrict to the range of parameters where the battery model is
physically realized, by assuming that $\rho$, $\mu$, $N$, $\ell$, and
$u_0$ are such that $u_\infty(x)>0$ for $0\le x\le \ell$, where
$u_\infty$ is defined in (\ref{uinftydef}).  Then the bound charge
profile $\{u(t,x),\,0\le x\le \ell,t\ge 0\}$ of the battery, defined
as the solution of the PDE (\ref{pdefinal}), is given for
  the case $\delta\not=0$ by
\begin{equation*}
u(t,x)=\int_0^\ell u_0(y)p_{\ell,\delta}(t,y,x)\,dy
-\int_0^t p_{\ell,\delta}(t-s,0,x)\,\Lambda(ds)
+\frac{\rho N}{\delta}\int_0^\ell
  (p_{\ell,\delta}(t,y,x)-\frac{1}{\ell})\,dy, 
\end{equation*}
and for the case $\delta=0$ by
\begin{align}\nonumber
u(t,x) &=\int_0^\ell u_0(y) p_{\ell,0}(t,y,x)\,dy
-\int_0^t p_{\ell,0}(t-s,0,x)\,\Lambda(ds)\\
&\quad +  \rho N(1-\frac{2x}{\ell})+ 4\rho  
\sum_{n=1}^\infty \cos(\frac{n\pi x}{\ell}) 
   \frac{((-1)^n -1)}{n^2\pi^2}e^{-2\kappa(n^2\pi^2)t}
\label{soldelta0}
\end{align}
where $p_{\ell,\delta}(t,y,x)$ is defined in (\ref{transitionprob}).
\end{theorem}

\begin{proof}  The special case of diffusion but no migration, which is the
  PDE (\ref{pdefinal}) with $\rho=0$, that is
\begin{eqnarray}\nonumber
&&du(t,x)=-\Lambda(dt)\delta_0(dx) +2\kappa\ell^2
\frac{\partial^2 u}{\partial x^2}(t,x)\,dt
+4\kappa\ell\mu \frac{\partial u}{\partial x}(t,x)\,dt, \quad 0\le
x\le \ell\\\nonumber
&& \quad \ell\, \frac{\partial u}{\partial x}(t,0+)=-2\mu u(t,0),
\quad \ell\, \frac{\partial u}{\partial x}(t,\ell-)
=2\mu u(t,\ell),\,\quad u(0,x)=u_0(x), 
\label{pdespecialcase1}
\end{eqnarray}
has been studied in Ref.\ \cite{kajkon2012}. The solution is given by 
\begin{equation}\label{eq:pdesolsum}
u(t,x)=
\int_0^\ell u_0(y)p_{\ell,\mu}(t,y,x)\,dy
-\int_0^t p_{\ell,\mu}(t-s,0,x)\,\Lambda(ds).
\end{equation}
To handle the case of a nonzero migration effect, $\rho\not=0$, in
(\ref{pdefinal}) we first observe that the non-homogeneous term in
(\ref{eq:pdesolsum}) which involves $\Lambda(dt)$ remains the
same.  Hence it suffices to discuss the solution of
of (\ref{pdefinal}) for the homogeneous case $\Lambda(dt)\equiv 0$,
which represents a battery at rest without discharge current. We claim
that for any drift parameters $\mu$ and $\rho$, such that
$\delta=\mu-\rho\not=0$, the solution is given by 
\begin{align}\nonumber
u(t,x)&=\frac{\rho N}{\rho-\mu}+ \int_0^\ell \Big(u_0(y)-\frac{\rho
  N}{\rho-\mu}\Big)p_{\ell,\mu-\rho}(t,y,x)\,dy\\
    &=\int_0^\ell u_0(y)p_{\ell,\delta}(t,y,x)\,dy
          +\frac{\rho N}{\delta} 
  \int_0^\ell (p_{\ell,\delta}(t,y,x)-\frac{1}{\ell})\,dy.
\label{pdesolhomogeneous}
\end{align}
Indeed, it is straightforward to verify that for any constant $M$ the function
\[
g(t,x)=M+ \int_0^\ell(u_0(y)-M)p_{\ell,\delta}(t,y,x)\,dy
\]
satisfies the target equation. Moreover,
\begin{align*}
\ell \frac{\partial g}{\partial x}(t,x)|_{x=0}&=2\delta(M-g(t,0))\\
&= -2\mu g(t,0)-2\rho((\rho-\mu)M/\rho-g(t,0)),
\end{align*}
which shows that boundary condition at $x=0$ is satisfied for $M=\rho
N/(\rho-\mu)$. Similarly for the boundary condition at
$x=\ell$. Finally, to cover the case $\delta=0$ one shows
that
\[
\lim_{\delta\to 0}
\frac{\rho N}{\delta}\int_0^\ell
  (p_{\ell,\delta}(t,y,x)-\frac{1}{\ell})\,dy
\]
equals the sum of the two last terms in (\ref{soldelta0}).
\end{proof}

\noindent
{\bf Constant initial charge.} 
Restricting to the special case of constant initial charge $u_0\equiv
N$, the solution for $\delta\not=0$ is
\begin{align*}
u(t,x)&=\Big(\frac{\mu N}{\delta}-\frac{1}{\ell}\Lambda(t)\Big)
\frac{2\delta \,e^{-2\delta x/\ell}}{1-e^{-2\delta}} -\frac{\rho N}{\delta}\\  
&\quad  + 4\mu N e^{-\delta x/\ell} 
\sum_{n=1}^\infty (\cos(\frac{n\pi x}{\ell})-\frac{\delta}{n\pi}\sin(\frac{n\pi
x}{\ell})) \frac{((-1)^n e^\delta-1)n^2\pi^2}{(\delta^2+n^2\pi^2)^2}
e^{-2\kappa(\delta^2+n^2\pi^2)t}\\
&\quad -\frac{2e^{-\delta x/\ell}}{\ell}
\sum_{n=1}^\infty (\cos(\frac{n\pi x}{\ell})-\frac{\delta}{n\pi}\sin(\frac{n\pi
x}{\ell})) \frac{n^2\pi^2}{\delta^2+n^2\pi^2}\int_0^t
e^{-2\kappa(\delta^2+n^2\pi^2)(t-s)}\,\Lambda(ds)
\end{align*}
and for $\delta=0$, 
\begin{align*}
u(t,x)&=N\Big(1+\rho-\frac{2\rho x}{\ell}\Big)  
-\frac{1}{\ell}\Lambda(t)
+ 4\mu N  \sum_{n=1}^\infty \cos(\frac{n\pi x}{\ell}) 
\frac{((-1)^n-1)}{n^2\pi^2} e^{-2\kappa n^2\pi^2\, t}\\
&\qquad -\frac{2}{\ell}
\sum_{n=1}^\infty \cos(\frac{n\pi x}{\ell})\int_0^t
e^{-2\kappa n^2\pi^2(t-s)}\,\Lambda(ds)
\end{align*}

\noindent
{\bf Asymptotic charge profile.}  For a battery at rest, so that
$\Lambda(t)\equiv 0$, we have by (\ref{homogeneousasymptotic}) and
(\ref{pdesolhomogeneous}) that $u_\infty(x)=\lim_{t\to\infty}u(t,x)$ is
given by
\begin{equation}\label{uinftydef}
u_\infty(x)=\frac{1}{\ell}\int_0^\ell u_0(y)\,dy\; 
\frac{2\delta\,e^{-2\delta x/\ell}}{1-e^{-2\delta}}
+\frac{\rho N}{\delta} \Big(\frac{2\delta\,e^{-2\delta
    x/\ell}}{1-e^{-2\delta}}-1\Big),\quad 0\le x\le\ell. 
\end{equation}
In particular, for $u_0(y)=N$, $\rho=0$ yields the truncated
exponential function
\[
u_\infty(x)=N \frac{2\mu e^{-2\mu x/\ell}}{1-e^{-2\mu}},
\quad 0\le x\le\ell,
\]
and $\mu=0$ the trivial asymptotic solution $u_\infty(x)=N$, $0\le
x\le\ell$.

\medskip
\noindent
{\bf Non-homogeneous case, migration but no diffusion}.  The system
(\ref{pdefinal}) for the case $\mu=0$ and arbitrary $\rho$, where
we also restrict to the initial condition $u_0(x)=N$, $0\le x\le
\ell$, takes the form
\begin{eqnarray}\nonumber
&&du(t,x)=-\Lambda(dt)\delta_0(dx) +2\kappa\ell^2
\frac{\partial^2 u}{\partial x^2}(t,x)\,dt
-4\kappa\ell\rho \frac{\partial u}{\partial x}(t,x)\,dt, \quad 0\le
x\le \ell\\\nonumber
&& \quad \ell\, \frac{\partial u}{\partial x}(t,0+)=
-2\rho(N-u(t,0))\\
&&\quad \ell\, \frac{\partial u}{\partial x}(t,\ell-)
=2\rho(N-u(t,\ell)),\,\qquad
u(0,x)=N. 
\label{pdespecialcase2}
\end{eqnarray}
The solution is
\[
u(t,x)=N-\int_0^t p_{\ell,-\rho}(t-s,0,x)\,\Lambda(ds).
\]

\medskip
\noindent
{\bf Varying the discharge profile}.  Of course, the result in
Theorem \ref{thm:pde} for the capacity reservoir $u(t,x)$ will be
obtained in a more or less explicit form depending on which discharge
profile $\Lambda$ applies.  But, in principle, the results of
Theorem \ref{thm:pde} allow comparison of the capacity dynamics under
the discharge patterns discussed in
Section \ref{sec:dischargeprofile}, and others.  For example, a family
of deterministic on-off patterns with given parameters may be compared
to the reference case of constant current discharge.  If the discharge
pattern is Poisson or otherwise random then the response in capacity
and voltage drop will be random as well.  The next subsection is
concerned with the case of constant discharge.

\subsection{Autonomous system}

The results in this subsection applies to the case of constant
current $\Lambda(t)=\lambda t$.  For simplicity we assume constant
initial capacity $u_0=N$. Then the available capacity
predicted by the spatial kinetic battery model according to Theorem
2.1 is
\begin{align}\nonumber
u(t,0)&=\frac{\mu N}{\delta}\,\omega_\delta 
-\frac{\rho N}{\delta}
+ 4\mu N  
\sum_{n=1}^\infty \frac{((-1)^n e^\delta-1)n^2\pi^2}{(\delta^2+n^2\pi^2)^2}
e^{-2\kappa(\delta^2+n^2\pi^2)t}\\
&\quad- \omega_\delta 
\,\frac{\lambda t}{\ell}
-\frac{\lambda}{\kappa\ell}
\sum_{n=1}^\infty  \frac{n^2\pi^2}{(\delta^2+n^2\pi^2)^2}
(1-e^{-2\kappa(\delta^2+n^2\pi^2)t}),\quad 
\omega_\delta=\frac{2\delta}{1-e^{-2\delta}}, 
\label{nomcapacityspatial} 
\end{align}
and the corresponding state of charge is $\wt u(t,0)=u(t,0)/N$.
It is illustrative to consider the available capacity
as a function of the remaining capacity stored in the
battery. This point of view is explicit in the simple kinetic battery model 
under constant current, namely if we consider $(v(t),u(t))$ with
$v(t)=T-\lambda t$ and 
\[
u(t)=N-c\lambda t-\lambda(1-c)(1-e^{-k_c(1+p)t})/k_c,\quad c=N/T,
\]
which is the solution of (\ref{kibameqnplussol}) for the case
$\Lambda(t)=\lambda t$.  Then $(v,u)$ with $u=u(v)$  given by  
\[
u=cv-\gamma_{c,p}(1-c)(1-e^{-(T-v)/\gamma_{c,p}}),\quad 
                \gamma_{c,p}=\frac{\lambda}{k_c (1+p)},
\]
is an autonomous system, and we may express functionals of $u$, such
as voltage $E=E(u)$, in terms of $v$.  In the same spirit we seek to
express $u(t,0)$ for the spatial kinetic battery model as a function
of remaining capacity.  In our model the relevant remaining capacity
function is what is stored in the entire reservoir of size $\ell$ at
time $t$, namely
\[
v(t)=u(t,0)+\int_0^\ell u(t,x)\,dx=u(t,0)+N\ell-\lambda t.
\]
By replacing $\lambda t$ in (\ref{nomcapacityspatial}) with
$u(t,0)-v(t)+N\ell$,
\begin{align}\nonumber
u(t,0)&= 
(v(t)-u(t,0))\frac{\omega_\delta}{\ell}
+\frac{\rho N}{\delta}(\omega_\delta-1) \\\nonumber
&\quad + 4\mu N \sum_{n=1}^\infty \frac{((-1)^n
  e^\delta-1)n^2\pi^2}{(\delta^2+n^2\pi^2)^2} 
e^{-2\kappa(\delta^2+n^2\pi^2)(u(t,0)-v(t)+N\ell)/\lambda}\\
&\quad-\frac{\lambda}{\kappa\ell}
\sum_{n=1}^\infty  \frac{n^2\pi^2}{(\delta^2+n^2\pi^2)^2}
(1-e^{-2\kappa(\delta^2+n^2\pi^2)(u(t,0)-v(t)+N\ell)/\lambda}),
\label{phaseplanespatial}
\end{align}
which is an autonomous system for the pair $(v,u)=(v(t),u(t,0))$.
Figure \ref{fig:phaseplanespatial} indicates the typical shape of
solution curves $(v,u)$ for fixed, but arbitrary, parameter values
$T=1000$, $N=100$, $\ell=(T-N)/N=9$, $\lambda=1000$, and $\kappa=0.5$,
and for varying $\rho$ and $\mu$.  The diagonal line from
  $(0,0)$ to $(T,N)$ in the $(v,u)$ phaseplane represents an idealized situation
  where all theretical capacity of the battery is consumed at a steady
  pace. The additional straight line from $(T-N,0)$ to $(T,N)$
  represents a battery which is drained at rate $\lambda$
  without any recovery of bound charge at all.
\begin{figure}[htb]
\centerline{\hbox{\epsfxsize=15.0 truecm
            \epsffile{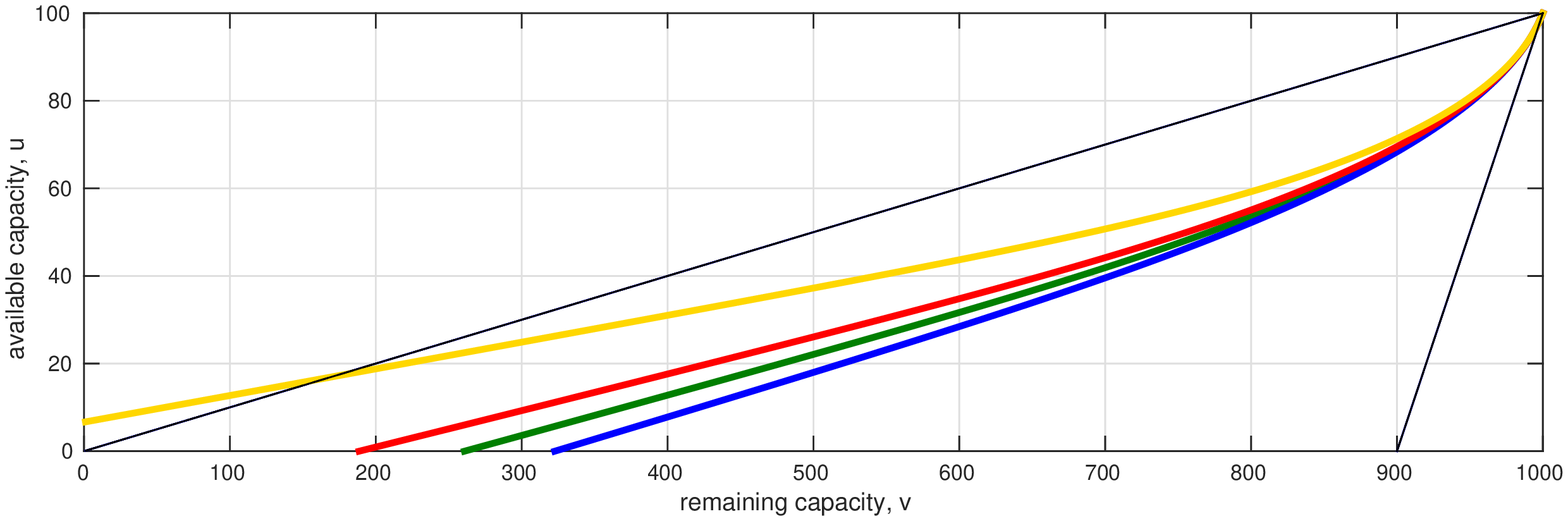}}}
\centerline{\hbox{\epsfxsize=15.0 truecm
            \epsffile{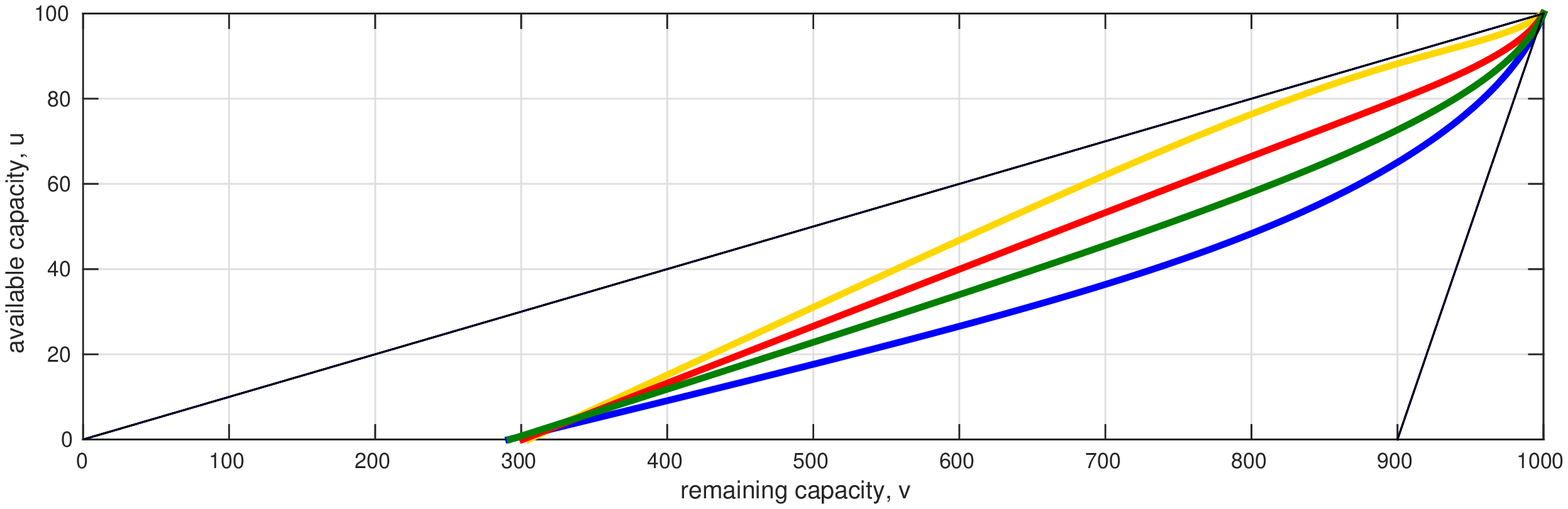}}}
\caption{Solution curves of (\ref{phaseplanespatial}) for $T=1000$,
  $N=100$, $\lambda=1000$, $\kappa=0.5$.  In the upper panel, the
  diffusion parameter is $\mu=0$ and the migration parameter varies
  from $\rho=0$ (blue curve), $\rho=0.1$ (green curve), $\rho=0.2$
  (red curve), $\rho=0.5$ (yellow curve).  The lower panel shows four
  different combinations of $\mu$ and $\rho$ which all give
  approximately the same utilization of the total available battery
  capacity, namely $\rho=0.1$, $\mu=-0.1$ (blue curve), $\rho=0$,
  $\mu=0.1$ (green curve), $\rho=-0.1$, $\mu=0.25$ (red curve), and
  $\rho=-0.2$, $\mu=0.4$ (yellow curve).}
\label{fig:phaseplanespatial}
\end{figure}

\section{Nonlinear Battery Models}

In this section we introduce an approach to nonlinear battery modeling
using a discrete time Markov chain on a
set of bivariate states representing available and remaining
capacities.  The chain has nonlinear jump probabilities which we
arrive at by analyzing a simplified transport system of charges under
diffusion and migration. Then we investigate the scaled
capacities as the number of slots per time unit tends to infinity and
the discharge process converges to that of constant rate. In the
limit we obtain deterministic capacity functions identified as the
solution of a nonlinear ordinary differential equation. Informed by
these findings we then consider a class of nonlinear ODEs which can be
solved explicitly. The solutions arising in this manner potentially
reflect the nonlinear dynamics of capacity under discharge of the
battery. In addition we indicate two further directions of
probabilistic modeling. One is to study the deviation of the Markov
chain from its deterministic limit and describe the scaled
fluctuations in terms of a diffusion process. Finally we consider the
deterministic limit process modified to operate under random
discharge, and compare this situation with the previous cases.

\subsection{Markov chain model in slotted time}\label{sec:markov}

We consider a discrete time Markov chain $(V_n,X_n)_{n\ge 0}$ defined
on the state space $E=[0,T]\times[0,N]$, modeling
\begin{align*}
V_n&=\mbox{remaining capacity [Ah] in time slot $n$}\\
X_n&=\mbox{available capacity [Ah] in time slot $n$.} 
\end{align*}
We assume that $N$ and $T$ are integer multiples of $\delta$ and that
all jumps are of size $\delta$. All jumps in the first coordinate are
downwards. Jumps in the second are allowed to be both up and down as
long as $X_n\le N$.  The battery is discharged randomly at constant
current $\delta$ with probability $q$ per slot. Letting
\[
Z_n=\mbox{number of load units discharged in slot $n$}, \quad n\ge 1.
\]
where $\{Z_i\}$ is a sequence of i.i.d random variables with
$P(Z_i=1)=1-P(Z_i=0)=q$, it follows that $\Lambda_n=\delta
\sum_{i=1}^n Z_i$ is the accumulated discharge at slot $n$ and
$V_n=T-\Lambda_n$ is the remaining charge in the battery after $n$
slots. The expected discharge rate is
$\bar\lambda=E(\Lambda_n)/n=\delta q$.  Given the sequence
$(V_n)_{n\ge 0}$ as input we model $(X_n)_{n\ge 0}$ as a Markov chain
modulated by $(V_n)$.  All jumps down of $(X_n)$ are inherited from
the discharge profile and follow those of $(V_n)$.  Jumps up will
occur according to a Markovian dynamics chosen so as to reflect the
recovery properties of the battery. In slot $n+1$ the transition
probabilities depend on the current state $X_n$ and the current
discharge information stored in $V_n$.

\subsubsection*{Transport system.}  
In an attempt to model charge recovery the battery cell is thought to
consist of a randomly structured,
electroactive material which allows transport of charge carrying
species through the electrolyte by liquid or solid state diffusion. 
Internal charge recovery relies on access to transportation channels
of enough connectivity to allow the material to pass from one node to
the other. Also, these channels must be ``activated'' by a sufficient
amount of previous discharge events. To try to describe such a
system, we introduce
\begin{align*}
K_n&=\mbox{available concentration of charge in slot $n$}\\
L_n&=\mbox{number of charge-carrying migration channels in slot $n$}\\
M_n&=\mbox{number of channels in slot $n$ activated by
  electrons at the cathode.}
\end{align*} 
Conditional on $(V_n,X_n)$, the updates $K_{n+1}$ and $L_{n+1}$ are
assumed to have Poisson distributions, such that for given nonnegative
parameters $\alpha>0$ and $\beta>0$
\begin{align*}
K_{n+1}\in \mathrm{Po}(V_n,\beta),\quad L_{n+1}\in {\rm
  Po}(\alpha(N-X_n)),
\end{align*}
and $M_{n+1}$ is binomially sampled from $L_{n+1}$, so that
\[  
M_{n+1}\in {\rm Bin}(L_{n+1},q)\stackrel{d}{=}{\rm Bin}(N-X_n,q \alpha)
\stackrel{d}{\sim} {\rm Po}(q\alpha\,(N-X_n)).
\]
In case there is no discharge in slot $n$, that is $Z_n=0$, then the
battery cell is able to recover one unit of charge if both 
$K_n\ge 1$ and $M_n\ge 1$.  Hence
\begin{align*}
V_{n+1}&=V_n-\delta Z_{n+1}\\
X_{n+1}&=X_n-\delta Z_{n+1}+\delta (1-Z_{n+1}){\mathbf 1}_{\{K_{n+1}\ge
  1,\,M_{n+1}\ge 1\}}
\end{align*}
The dynamics specified by this recursive relation is that, given
$(V_n,X_n)=(v,x)$, if $Z_{n+1}=1$ then the transition in slot
$n+1$ is $(v,x)\to (v-\delta,x-\delta)$ and if $Z_{n+1}=0$ then
\[
(v,x)\to \left\{
\begin{array}{ccc}
(v,x+\delta) & \mbox{with probability} &  (1-e^{-\beta
    v})(1-e^{-\alpha q(N-x)})\\
  (v,x)   &  \mbox{-''-} & 1-(1-e^{-\beta v})(1-e^{-\alpha q(N-x)}).
\end{array}
\right.
\]
Together these relations define a bivariate Markov chain model
$(V_n,X_n)_{n\ge 0}$ with dynamics specified by 
\[
(v,x)\to \left\{
\begin{array}{ccc}
(v-\delta,x-\delta) & \mbox{with prob.} & q\\
(v,x+\delta) & \mbox{-''-} &  (1-q)(1-e^{-\beta v})(1-e^{-\alpha q(N-x)})\\
  (v,x)   &  \mbox{-''-} & (1-q)(1-(1-e^{-\beta v})(1-e^{-\alpha q(N-x)}))
\end{array}
\right.
\]
and, typically, initial condition $(V_0,X_0)=(T,N)$.  We obtain a
drift function and a variance function for the Markov chain from
\[
E(X_{n+1}-X_n|(V_n,X_n))=-q\delta+(1-q)\delta(1-e^{-\beta
  V_n})(1-e^{-\alpha q(N-X_n)})
\]
and 
\[
E((X_{n+1}-X_n)^2|(V_n,X_n))=q\delta^2+(1-q)\delta^2(1-e^{-\beta V_n})(1-e^{-\alpha q(N-X_n)})
\]

\subsection{Continuous time approximation}\label{sec:conttimeapprox}

The drift function $m(v,x)=E(X_{n+1}-X_n|V_n=v,X_n=x)$ suggests a
relevant, approximating ODE for the Markov chain. To formalize this
limit procedure it is convenient to introduce a scaling parameter
$m\ge 1$. At scaling level $m$ the number of slots per unit time is
$m$ and the discharge current jump size is $\delta/m$ rather than
$\delta$.  Consider
\[
X^m_n=\mbox{nominal capacity in slot $n$ at scaling level $m$}
\]
and define for continuous time $t\ge 0$,
\[
X^{(m)}(t)=X^m_{[mt]}. 
\]
Similarly, let $\Lambda^{m}_n$ be the scaled discharge process
with $\delta$ replaced by $\delta/m$ and put
\[
\Lambda^{(m)}(t)=\Lambda^m_{[mt]}, \quad
V^{(m)}(t)=T-\Lambda^{(m)}(t).
\]  
Then 
\[
\Lambda^{(m)}(t)
=\frac{\delta}{m}\sum_{k=1}^{[mt]}Z_i
=\frac{q\delta[mt]}{m}+\sqrt{\frac{q(1-q)\delta^2}{m}}  
\frac{1}{\sqrt{m}}\sum_{k=1}^{[mt]}\frac{Z_i-q}{\sqrt{q(1-q)}}.
\]
By the functional central limit theorem we may introduce a Wiener
process $W_1(t)$ and for large $m$ view $V^{(m)}(t)$ 
as an approximation of the continuous time remaining capacity
function $V_t=T-\lambda t$, in the sense
\[
dV^{(m)}(t)=-\lambda \,dt+\frac{\sigma}{\sqrt{m}}\,dW_1(t),\quad V^{(m)}(0)=T,
\]
where $\sigma^2=q(1-q)\delta^2$ and the approximation error is of the
order $1/\sqrt{m}$.  Moreover, by considering the differential change
of $X^{(m)}(t)$ over a time interval $(t,t+h)$ where $h=1/m$,
\begin{align*}
E(X^{(m)}(t+h)-X^{(m)}(t)|&(V^{(m)}(t),X^{(m)}(t))=(v,x))\\
&=h \big(-\lambda+(1-q)\delta(1-e^{-\beta v})(1-e^{-\alpha q(N-x)})\big)
\end{align*}
and 
\begin{align*}
E((X^{(m)}(t+h)-X^{(m)}(t))^2|&(V^{(m)}(t),X^{(m)}(t))=(v,x))\\
&=h \frac{1}{m}\big(\lambda\delta+(1-q)\delta^2(1-e^{-\beta v})\,(1-e^{-\alpha q(N-x)})\big).
\end{align*}
As above we obtain a deterministic limit equation for large $m$ by
applying a diffusion approximation with the diffusion term of magnitude
$1/\sqrt{m}$.  To simplify notation we put
\[
f(v,x)=\delta(1-q)(1-e^{-\beta v})\,(1-e^{-\alpha q(1-x)}).
\]
Then 
\begin{align*}
&dX^{(m)}(t)=-\lambda\,dt+f(V^{(m)}(t),X^{(m)}(t))\,dt
+\sqrt{\frac{\lambda\delta}{m}+\frac{\delta}{m}
  f(V^{(m)}(t),X^{(m)}(t))}\,dW_2(t), 
\end{align*}
where $W_2$ is another Wiener process. Since $V^{(m)}$ and $X^{(m)}$
have simultaneous jumps, $W_1$ and $W_2$ are dependent with a non-zero
covariance.

\subsection{Deterministic approximation of the Markov chain model}
\label{sec:approxmarkov}

As $m\to\infty$, the stochastic differential equations for $V^{(m)}$
  and $X^{(m)}$ simplify and become the ordinary differential equation
\begin{equation} \label{eq:conttimeapproxode}
\begin{array}{lll}  
&x_t'=-\delta q+(1-q)\delta(1-e^{-\beta v_t})(1-e^{-\alpha q(N-x_t)}),&\qquad x_0=N, \\[2mm]
&v_t'=-\delta q,&\qquad v_0=T.
\end{array}
\end{equation}
Recalling $\lambda=\delta q$, the solution is $v_t=T-\lambda t$ and 
\[
x_t=N-\frac{1}{\alpha q}\ln \Big(1-\alpha q 
\int_0^t\frac{h(s)}{h(t)}\,\lambda ds\Big), 
\]
where
\[
h(t)/h(0)=
\exp\Big\{\lambda \alpha t-(1-q)\alpha
e^{-\beta T}(e^{\lambda\beta t}-1)/\beta\Big\}.
\]
Hence 
\[
\alpha q \int_0^t\frac{h(s)}{h(t)}\,\lambda ds
=\alpha q\int_0^{\lambda t} 
\exp\Big\{-\alpha s+(1-q)\alpha e^{\lambda\beta t-\beta T}(1-e^{-\beta
  s})/\beta\Big\}\,ds    
\]
and so 
\begin{equation}\label{markovode}
x_t=N-\frac{1}{\alpha q}\ln \Big(1-\alpha q 
\int_0^{T-v_t} \exp\Big\{-\alpha s+(1-q)\alpha 
e^{-\beta v_t}(1-e^{-\beta s})/\beta\Big\}\,ds \Big).
\end{equation}
Thus, in analogy with the results obtained for the kinetic battery
models in the previous section, it follows that $(v,x)=(v(t),x(t))$
is an autonomous system from which we can read off the available
capacity as a function of remaining capacity.

\subsubsection*{On-off discharge pattern}

In the Markov chain model of section \ref{sec:markov}, we considered
random discharge at current $\delta$ with probability $q$ per slot,
which converged 
by the law of large numbers to a constant discharge
pattern $\Lambda(t)=\lambda t$, $\lambda=\delta q$, under the
approximation scheme of section \ref{sec:approxmarkov}.  An
alternative would be to run the Markov chain $(V_n,X_n)$ relative to a
given discharge sequence $(Z_n)$, such that in the scaling limit
emerges an on-off discharge pattern as discussed in
section \ref{sec:dischargeprofile}.  We recall that $(J_t)$ is the
piecewise constant indicator function which is one during periods when
the battery is under load and zero otherwise and that
$V_t=T-\delta\int_0^t J_u\,du$ is the remaining capacity at time
$t$. As the degree of resolution increases by taking $m\to \infty$ we
then expect the process of nominal capacity, $(X^{(m)})_{t\ge 0}$, to converge
to a deterministic limiting function $(X_t)_{t\ge 0}$, which solves
the ordinary differential equation
\[
dX_t=-\delta J_t\,dt+\delta(1-J_t)(1-e^{\alpha q(N-X_t)})(1-e^{-\beta V_t})\,dt.
\]
Then 
\[
e^{\alpha q(N-X_t)}\,dX_t=\delta(e^{\alpha q(N-X_t)}-1)
\{-J_t+(1-J_t)(1-e^{-\beta V_t})\}\,dt -\delta J_t\,dt
\]
and so, by introducing a generating factor $H_t$ defined by
\[
\frac{d}{dt}\ln H_t=\alpha q\delta(-J_t+(1-J_t)(1-e^{-\beta V_t})),
\]
hence
\[
H_t=H_0\exp\Big\{\alpha q\delta\int_0^t (-J_u+(1-J_u)(1-e^{-\beta V_u}))\,du\Big\},
\]
we can solve for $X_t$ and obtain
\[
X_t=N-\frac{1}{\alpha q}\ln\Big(1+\frac{\alpha q\delta}{H_t} \int_0^t
J_sH_s\,ds\Big)
\]
With $\lambda=q\delta$ this may be written
\begin{align*}
X_t=N-\frac{1}{\alpha q}\ln\Big(1+ \lambda \alpha \int_0^t J_s\,
e^{\lambda\alpha \int_s^t(J_u-(1-J_u)(1-e^{-\beta V_u}))\,du}\,ds
\Big), 
\end{align*}
which is a closed form solution for capacity in terms of the given
discharge profile, coded by $(J_t)$ and $(V_t)$ with parameters $\delta$
and $q$, and the additional battery parameters $\alpha$ and $\beta$.

\subsection{Deviation from deterministic behavior}

We have found a deterministic function $(v_t,x_t)$ which satisfies the
ordinary differential equation (\ref{eq:conttimeapproxode}), in the
limit $m\to\infty$ of the scaled Markov chain
$(V^{(m)}(t),X^{(m)}(t))_{t\ge 0}$ with constant discharge rate
$\lambda$.  Now we introduce the random quantities
\[
{\mathcal V}_m(t)=\sqrt{m}(V^{(m)}(t)-v_t),
\quad {\mathcal X}_m(t)=\sqrt{m}(X^{(m)}(t)-x_t)
\]
in order to study the fluctuations of the scaled Markov chain around
its deterministic limit $(v_t,x_t)$.  With $W_1$ and $W_2$ as in
section \ref{sec:conttimeapprox} we obtain for large $m$,
\begin{align*}
d{\mathcal V}_m(t)&=\sigma \,dW_1(t)\\
d{\mathcal X}_m(t)&=\sqrt{m}(f(V^{(m)}(t),X^{(m)}(t))-f(v_t,x_t))\,dt\\
&\quad +\sqrt{\lambda\delta+\delta f(V^{(m)}(t),X^{(m)}(t))}\,dW_2(t),
\end{align*}
where ${\mathcal V}_m(0)={\mathcal X}_m(0)=0$.  It follows by a Taylor
expansion of $f(v,x)$ that
\[
\sqrt{m}(f(V^{(m)}(t),X^{(m)}(t))-f(v_t,x_t))
\approx  f'_v(v_t,x_t){\mathcal V}_m(t)\,dt+f'_x(v_t,x_t){\mathcal X}_m(t)
\]
In the scaling limit $m\to\infty$ we therefore expect that the fluctuation
process ${\mathcal X}_m(t)$ converges to a diffusion process
${\mathcal X}(t)$, such
that 
\begin{align*}
d{\mathcal X}(t)=f'_v(v_t,x_t)\sigma W_1(t)\,dt+f'_x(v_t,x_t){\mathcal
  X}(t)\,dt
+\sqrt{\lambda\delta+\delta f(v_t,x_t)}\,dW_2(t).
\end{align*}
This stochastic differential equation for ${\mathcal X}(t)$ can be seen as a
generalized Ornstein Uhlenbeck process 
\[
d{\mathcal X}(t)=a_t W_1(t)\,dt+b_t\,{\mathcal X}(t)\,dt+c_t\,dW_2(t),
\]
with time-inhomogeneous drift and variance coefficients $b_t$  and
$c_t$, which is also modulated by an additional, random, drift $a_t
W_1(t)$.

\subsection{Other versions of nonlinear ODE battery capacity models}
\label{sec:nonlinearode}

To derive the capacity dynamics in (\ref{markovode}) we were guided by
a Markov chain argument based on a simplified view of charge transport
inside the battery cell.  In this final section we mention a more
general class of deterministic non-linear recovery models for the
dynamics of nominal capacity and state of charge. Again we write $x_t$
for nominal capacity and $v_t$ for the remaining capacity in the cell
as functions of time $t\ge 0$. Here $v_t=T-\Lambda(t)$ and
$\Lambda(t)$ is the discharge process with average discharge current
$\bar\lambda=\lambda$.  

The general principle for charge recovery is that of balancing the
discharge rate in $x_t$ by a positive drift of the nominal capacity
due to the release and transport of stored charges.  It is reasonable
that such effects are proportional to the applied average load $\bar
\lambda$ and exist as long as the theoretical capacity of the cell has
not yet been fully consumed, that is $v_t\ge 0$.  Furthermore, the
gain in capacity due to recovery depends on the migration of
charge-carriers, and the strength of this effect should increase with
the gap $N-x_t$ between maximal and actual capacity.  Hence, to
capture charge recovery in the framework of one-dimensional
differential equations, we may let $F:[0,\infty)\to [0,1]$ and
  $G:[0,\infty)\to [0,1]$ be non-decreasing functions with $F(0)=0$,
    $G(0)=0$, and consider equations of the generic shape
\begin{equation}\label{genericode}
dx_t=-\Lambda(dt)+\lambda F(N-x_t)G(v_t)\,dt, \quad t\le t_0,\quad x_0=N,
\end{equation}
where $t_0$ is battery life defined as the maximal time $t$ for which
$x_t\ge 0$.  

The special case of (\ref{genericode}) where $F(u)=1-e^{-au}$ and $a$
is a positive constant which measures the strength of migration of
charges due to the existing electric field in the battery, is given by
\begin{equation}\label{expode}
dx_t=-\Lambda(dt) + \lambda (1-e^{-a(N-x_t)})G(v_t)\,dt, \quad X_0=N. 
\end{equation}
Because of the separation of variables we may write 
\[
d h_t(e^{a(N-x_t)}-1)=a\int_0^t h_s\,\Lambda(ds),\quad 
\ln(h_t/h_0)=-a\Lambda(t)+\lambda a\int_0^tG(v_s)\,ds.
\]
For example, taking the simplest case $\Lambda(t)=\lambda t$,
\[
x_t=N-\frac{1}{a}\ln\Big(1+ \lambda a\int_0^t
e^{\lambda a\int_s^t (1-G(v_u))\,du}\,ds\Big).
\]
Equivalently, 
\[
x_t=N-\frac{1}{a}\ln\Big(1+ a\int_0^{T-v_t}
\exp\Big\{a\int_{v_t}^{v_t+s} (1-G(u))\,du\Big\}\,ds\Big).
\]
It is immediate in this model that the pair $(v_t,x_t)$ is autonomous
so that available capacity $x=x(v)$ can be viewed as a function of
remaining capacity only. Moreover, the capacity dynamics is load-invariant in
the sense that the curve $(v,x(v))$ does not depend on $\lambda$. In
other words, batteries drained using different choices of $\lambda$
will exhibit different battery life, longer the smaller intensity of
the current, but the used capacity $T-v$ at end of life will be the
same in each case.

\subsubsection*{Performance measures}

In the framework of the nonlinear ODE approach we have obtained, just
as for the linear models studied previously, a phase plane relation
$(v,x)=(v_t,x_t)$ for the capacity dynamics. Via state-of-charge $\wt
x=x/N$ we may proceed as before to modeling the corresponding voltage.
Without going into details this will give us some capacity threshold
$x_0$ below which the battery is no more functioning.  Then the unused
capacity that remains in the battery at the end of its life time is
the unique solution $v_0$ of $x(v)=x_0$.  Thus, the delivered capacity
is $D=T-v_0$ and the gained capacity is $G=D-N$.  Our model allows for
some qualitative conclusions about these quantities as well as
numerical studies of special cases. Various discharge processes
$\Lambda$ can be compared against experimental data and relevant 
parameters estimated.  To give an indication of this
type of work, Figure \ref{fig:dischargeprofile} shows the result of
repeated independent simulations based on (\ref{expode}) with
$G(v)=1-e^{-\beta v}$ and $\Lambda$ a random Poisson discharge
process as described in Section \ref{sec:dischargeprofile}. The
parameters used for these particular simulations are chosen
arbitrarily such that the output appears to mimic that of a real
battery.  The upper panel shows the phase plane traces of the resulting
solutions $(v,x)$. The lower panel shows the corresponding random
paths of the state-of-charge $\wt x_t$ as function of time.  Also in
the lower panel we have superimposed (solid red curve) the
state-of-charge for the case of constant discharge with the same
average load, now given by the explicit solution of equation
(\ref{expode}) with $\Lambda(t)=\lambda t$.
\begin{figure}[htb]
\centerline{\hbox{\epsfxsize=11.0 truecm
            \epsffile{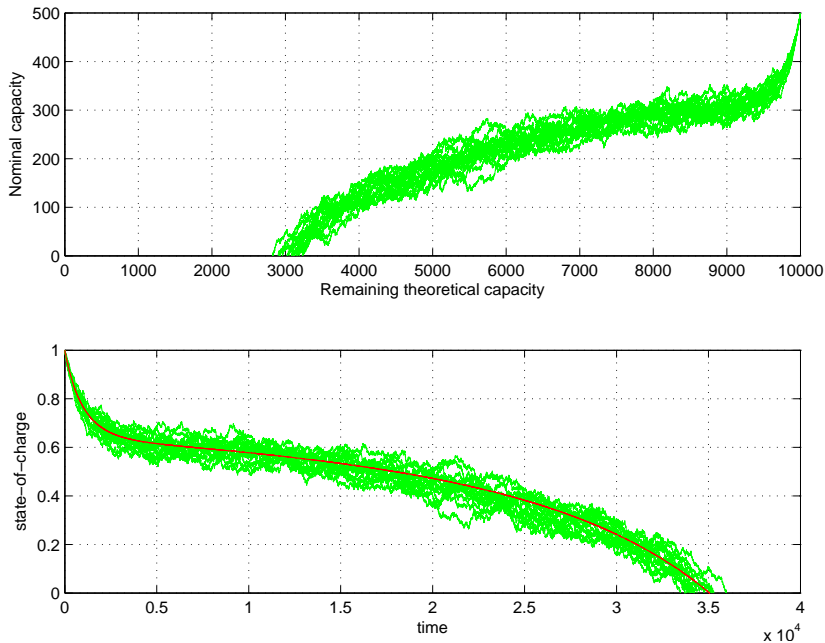}}}
\caption{Upper panel: Phase plane solutions $(v,x)$ of
  (\ref{expode}) with $G$ exponential and $\Lambda$ simulated
  Poisson processes.  Lower panel: Corresponding simulated paths of
  state-of-charge $\widetilde x_t$ as functions of time.  The
  superimposed smooth curve is the deterministic solution for constant
  discharge of the same average rate.} 
\label{fig:dischargeprofile}
\end{figure}

\end{document}